\documentclass[11pt]{amsart}
\usepackage[T1]{fontenc}
\usepackage{amsmath, amssymb, amsthm}
\usepackage{mathtools}
\usepackage[hypertexnames=false]{hyperref}%ないとだめ
\usepackage{cleveref}
\usepackage{autonum}
\usepackage{booktabs}
\usepackage{longtable}
\usepackage{float}
 \allowdisplaybreaks
\newtheorem{thm}{Theorem}[section]
\newtheorem{lem}{Lemma}[section]
\newtheorem{prop}{Proposition}[section]
\newtheorem{cor}{Corollary}[section]
\newtheorem{eg}{Example}[section]
\newtheorem{definition}{Definition}[section]
\newtheorem{q}{Question}[section]
\newtheorem{rem}{Remark}[section]

\newtheorem{claim}{Claim}[section]
\newtheorem{prob}{Problem}[section]

\crefname{thm}{Theorem}{Theorem}
\crefname{lem}{Lemma}{Lemma}
\crefname{prop}{Proposition}{Proposition}
\crefname{cor}{Corollary}{Corollary}
\crefname{eg}{Example}{Example}
\crefname{definition}{Definition}{Definition}
\crefname{q}{Question}{Question}
\crefname{rem}{Remark}{Remark}
\crefname{conj}{Conjecture}{Conjecture}
\crefname{claim}{Claim}{Claim}
\crefname{prob}{Problem}{Problem}

\def\fr#1{\mathfrak{#1}}
\def\bb#1{\mathbb{#1}}

\def\ali#1{\begin{align}#1\end{align}}
\def\dis#1{\displaystyle{#1}}
\def\w{\wedge}
\def\bw{\bigwedge}
\def\o{\oplus}
\def\bo{\bigoplus}
\def\ang#1{\langle #1\rangle}
\def\dim#1{{\rm dim\ }#1}
\def\0{\{ 0\}}
\def\Ker#1{{\rm Ker}\ #1}
\def\Im#1{{\rm Im}\ #1}

\title[The nilpotent fundamental groups]{On the torsion-free nilpotent fundamental groups of smooth quasi-projective varieties of rank up to seven}
\author{Taito Shimoji}
\date{}
\address{Department of Mathematics,  Graduate School of Science,  The University of Osaka,  Osaka,  Japan}
\email{u215629i@ecs.osaka-u.ac.jp}
\begin{document}
\maketitle
\begin{abstract}
Let $M$ be a smooth quasi-projective variety. Assume that the (topological) fundamental group $\pi_1(M,x)$ is torsion-free nilpotent. We show that if the first Betti number $b_1(M) \le 3$,  then $\pi_1(M,x)$ is isomorphic to either $\mathbb{Z}^n$ for $n = 1,  2,  3$,  a lattice in the Heisenberg group $H_3(\mathbb{R})$, or a lattice in $\mathbb{R} \times H_3(\mathbb{R})$. Moreover,  we prove that $\pi_1(M,  x)$ is abelian or $2$-step nilpotent if its rank is less than or equal to seven. Moreover, we classify the fundamental groups of rank up to six and determine the case of rank seven when $b_1(M)=4$. Our main results provide a partial positive answer to a question on nilpotent (quasi-)K\"ahler groups posed by Aguilar and Campana.
\end{abstract}
\section{Introduction}
\subsection{Aguilar and Campana's question for the nilpotent fundamental groups of smooth quasi-projective varieties}

The following question is given by Aguilar and Campana in \cite{RodCam}.
\begin{q}
Let $G$ be a torsion-free nilpotent Kähler group. Can its
 nilpotency class $\nu (G)$ be at least $3$? Can $\nu (G)$ be arbitrarily large?
 Same questions for the quasi-Kähler groups.
\end{q}
This question is mainly motivated by the previous works of Campana and others (for example,  see \cite{Cam}, \cite{CT95}, \cite{Cam2}, \cite{AraNori} and \cite{CadDengYam}). We focus on the case where the algebraic varieties are smooth quasi-projective. Thus we consider the following problem. 
\begin{prob}\label{Campana}
Let $M$ be a smooth quasi-projective variety over $\bb{C}$. Assume that the (topological) fundamental group $\pi_1(M, x)$ is torsion-free nilpotent. Is $\pi_1(M, x)$ abelian or $2$-step nilpotent?
\end{prob}
In fact,  all the known examples of the nilpotent fundamental groups of smooth quasi-projective varieties are either abelian or $2$-step nilpotent (see \cite{Pierre},  \cite{Cam}, \cite{JohVen} and \cite{CadDengYam}). 
 
\subsection{Main theorems}
Let $M$ be a smooth quasi-projective variety over $\bb{C}$. The first Betti number $b_1(M)$ is given by 
\ali{b_1(M):=\dim H^1(M, \bb{R})}
where $H^1(M, \bb{R})$ is the first real cohomology of $M$.
Our first main theorem determines which finitely generated torsion-free nilpotent groups can occur as fundamental groups of smooth quasi-projective varieties with small first Betti number.
\begin{thm}[\protect\cref{b1<=3}]\label{b1}
Let $M$ be a smooth quasi-projective variety over $\bb{C}$. If $b_1(M)\leq 3$ and $\pi_1(M, x)$ is torsion-free nilpotent,  then $\pi_1(M, x)$ is abelian or $2$-step nilpotent. More precisely,  if $b_1(M)\leq 3$ and $\pi_1(M, x)$ is torsion-free nilpotent,  then $\pi_1(M, x)$ is isomorphic to one of the groups of the following  list.
\ali{
&\bullet\ \bb{Z} &&\text{ if }b_1(M)=1. \\ 
&\bullet\ \bb{Z}^2 \text{ or a lattice in }H_3(\bb{R}) &&\text{ if }b_1(M)=2.\\
&\bullet\ \bb{Z}^3\text{ or a lattice in }H_3(\bb{R})\times\bb{R} &&\text{ if }b_1(M)=3.
}
where $H_3(\bb{R})$ is the $3$-dimensional Heisenberg group.
\end{thm}
Let $\Gamma$ be a torsion-free $s$-step nilpotent group. The \textit{rank} $rk(\Gamma )$ of $\Gamma$ is defined as 
\ali{
rk(\Gamma ):=\sum_{i=1}^{s}rk(C^i\Gamma /C^{i+1}\Gamma )
}
where $C^i\Gamma$ is the lower central series and $C^{s+1}\Gamma=\{ e\}$. Note that \ali{
b_1(M)\leq rk(\pi_1(M, x)).
}
It is well known that every finitely generated torsion-free nilpotent group $\Gamma$ embeds as a lattice—that is, a discrete cocompact subgroup—in a unique simply connected nilpotent real Lie group $N_\Gamma$ (see \cite{Malcev,Raghu}). Moreover, the rank of $\Gamma$ coincides with the dimension of $N_\Gamma$. This leads us to ask which lattices in simply connected nilpotent Lie groups can occur as fundamental groups of smooth quasi-projective varieties. When the rank of the fundamental group—or, equivalently, the dimension of the associated Lie group—is small, we obtain the following results.
\begin{thm}[\cref{rk<=7}]\label{main3}
Let $M$ be a smooth quasi-projective variety. Assume that $\pi_1(M, x)$ is torsion-free nilpotent. If $rk(\pi_1(M, x))\leq 7$,  then $\pi_1(M, x)$ is abelian or $2$-step nilpotent.
\end{thm}%rank6までは分類できるというcorを述べる。５次元ハイゼンベルグも忘れずに
The proof of the above theorem follows from Morgan's condition on the bigraded Lie algebras (\cite[Theorem {\bf (9.4)}]{Morgan}). Using the condition of bigraded Lie algebras,  we determine the simply-connected real nilpotent Lie groups in which the torsion-free nilpotent fundamental groups can be embedded as lattices. The following theorems, \cref{prop1} and \cref{prop2} are the cases of $rk(\pi_1(M, x))\leq 6$ and $rk(\pi_1(M, x))=7$ respectively.

In order to state \cref{prop1}, we define the Lie group $N_6(t)$. For $t\in\bb{R}$, let $N_6(t)$ denote the simply connected nilpotent real Lie group whose Lie algebra is given by
\ali{
\fr{n}_6(t)=\ang{X_1,X_2,X_3,X_4,X_5,X_6}:[X_1,X_2]=[X_3,X_4]=X_5,[X_1,X_3]=X_6,[X_2,X_4]=tX_6.
}
The three Lie algebras $\mathfrak n_6(t)$, for $t=0,\pm1$, are pairwise non-isomorphic over $\bb{R}$. In particular, $\mathfrak n_6(1)$ is isomorphic to the direct sum of two copies of the three-dimensional Heisenberg Lie algebra, whereas $\mathfrak n_6(-1)$ is isomorphic to the underlying real Lie algebra of the complex Heisenberg Lie algebra $\mathfrak h_3(\mathbb C)$.
\begin{thm}[\cref{rk<=6}]\label{prop1}
Let $M$ be a smooth quasi-projective variety. If $\pi_1(M, x)$ is torsion-free nilpotent and $rk(\pi_1(M, x))\leq 6$, then $\pi_1(M, x)$ is a lattice in one of the Lie groups of the following list:
\ali{
\ &\bullet\ \bb{R}&\text{ if }b_1(M)=1. \\
\ &\bullet\ \bb{R}^2\text{ or }H_3(\bb{R})&\text{ if }b_1(M)=2. \\
\ &\bullet\ \bb{R}^3\text{ or }\bb{R}\times H_3(\bb{R})&\text{ if }b_1(M)=3. \\
\ &\bullet\ \bb{R}^4,  \bb{R}^2\times H_3(\bb{R}), H_3(\bb{R})\times H_3(\bb{R}), H_5(\bb{R})\text{ or }N_6(t) (t=-1,0)&\text{ if }b_1(M)=4. \\
\ &\bullet\ \bb{R}^5,  \bb{R}^3\times H_3(\bb{R})\text{ or }\bb{R}\times H_5(\bb{R})&\text{ if }b_1(M)=5. \\
\ &\bullet\ \bb{R}^6&\text{ if }b_1(M)=6. 
}  
Here, $H_{2k+1}(\bb{R})$ is the $(2k+1)$-dimensional Heisenberg group.
\end{thm}
For each $t$, a smooth quasi-projective variety whose fundamental group is a lattice in $N_6(t)$ can be constructed as a principal $(\mathbb C^*)^2$-bundle over an abelian surface. More precisely, let $L_1$and $L_2$ be suitable line bundles on an abelian surface $A$ whose first Chern classes are linearly independent, and let $L_i^{\times}$ denote the complement of the zero section in $L_i$. The pullback of 
\[
L_1^{\times}\times L_2^{\times}\rightarrow A\times A
\]
along the diagonal map $A\rightarrow A\times A$ gives the desired principal $(\bb{C}^*)^2$-bundle. The number of Chern classes used in this construction coincides with the dimension of the center of $N_6(t)$. 

In dimension 7, the classification of real nilpotent Lie algebras is more complicated than in dimension 6, because there are many more isomorphism classes. Moreover, non-Heisenberg Lie algebras first occur when \(b_1(M)=4\). For this reason, in \cref{prop2} (\cref{rk<=7}), we determine the isomorphism classes only for the case corresponding to \(b_1(M)=4\).

We use the notation $\fr{n}^i_7$ for complex seven-dimensional nilpotent Lie algebras following the classification of Goze--Khakimdjanov \cite{Nil}.
\begin{thm}[\cref{rk<=7}]\label{prop2}
Let $M$ be a smooth quasi-projective variety. Assume that $\pi_1(M, x)$ is torsion-free nilpotent and $rk(\pi_1(M, x))=7$. Let $N$ be the simply-connected nilpotent Lie group such that $\pi_1(M,x)$ is isomorphic to a lattice in $N$. Then $\pi_1(M, x)$ is abelian or $2$-step nilpotent. In particular, if $b_1(M)=4$, then the complexification of the Lie algebra $\fr{n}$ of $N$ is isomorphic to one of the following Lie algebras $\fr{n}_7^{142}$ or $\fr{n}_7^{143}$.
\ali{
\fr{n}_7^{142}=\ang{X_1, \dots, X_7}: [X_1, X_i]=X_{i-1}\text{ for i=3, 5, 7 }, [X_3, X_5]=X_4, [X_5, X_7]=X_2
}
and
\ali{
\fr{n}_7^{143}=\ang{X_1, \dots, X_7}: [X_1, X_i]=X_{i-1}\text{ for i=3, 5, 7 }, [X_3, X_5]=X_6, [X_5, X_7]=X_2
.}
\end{thm}
Note that the Lie algebras $\fr{n}_7^{142}$ and $\fr{n}_7^{143}$ are not isomorphic to the $(2k+1)$-dimensional Heisenberg Lie algebra for any $k$ since the dimensions of the derived ideals $[\fr{n}_7^{\bullet},\fr{n}_7^{\bullet}]$ are not $1$. A smooth quasi-projective variety realizing this case can likewise be constructed as a principal $(\mathbb C^*)^3$-bundle over an abelian surface (see \cref{egdim7}).

Our main results \cref{b1}, \cref{prop1} and \cref{prop2} are obtained by classifying low-dimensional nilpotent Lie algebras satisfying Morgan’s condition {\bf (W)}, which we now recall. 
\begin{thm}\cite[Theorem {\bf (9.4)}]{Morgan}\label{WW}
Let $M$ be a smooth quasi-projective variety. Let $N$ be a simply-connected real nilpotent Lie group. Assume that $\pi_1(M, x)$ is a lattice in $N$. For the Lie algebra $\fr{n}$ of $N$,  there exists a Hodge-theoretic bigrading $\fr{n}_{\bb{C}}=\bo_{p, q\leq 0, p+q\leq -1}\fr{n}_{p, q}$ such that for the induced bigrading $H^*(\fr{n}_{\bb{C}})=\bo_{p, q\geq 0, p+q\geq 1}H_{p, q}^*$, we have 
\ali{
&{\bf (W)}:
\begin{array}{l}
H^1(\fr{n}_{\bb{C}})
 = H_{1,0}^1(\fr{n}_{\bb{C}}) \o H_{0,1}^1(\fr{n}_{\bb{C}}) \o H_{1,1}^1(\fr{n}_{\bb{C}}) \text{ and }\\
H^2(\fr{n}_{\bb{C}})
 = H_{2,0}^2(\fr{n}_{\bb{C}}) \o H_{1,1}^2(\fr{n}_{\bb{C}}) \o H_{0,2}^2(\fr{n}_{\bb{C}})
 \o H_{2,1}^2(\fr{n}_{\bb{C}}) \o H_{1,2}^2(\fr{n}_{\bb{C}}) \o H_{2,2}^2(\fr{n}_{\bb{C}}) .
\end{array}
}
\end{thm}
Here, the term ``Hodge-theoretic bigrading'' is the bigradings induced by mixed Hodge structures (see \cite{Del71,Morgan}). 
\begin{rem}[\cref{8dim}]\label{example}
There exists a finitely generated torsion-free $3$-step nilpotent group such that the corresponding Lie algebra carries a mixed Hodge structure and satisfies the condition {\bf (W)}. Thus this group cannot be excluded by our techniques. Thus, in the case when $rk(\pi_1(M, x))=8$, we cannot conclude whether \cref{Campana} holds by using the method used in the proof of main theorems. The example is a lattice in the real Lie group whose Lie algebra is given by
\ali{
\fr{n}:=\ang{X_1, X_2, Y_1, Y_2, Z_1, Z_2, A, B}:&[X_1, Y_1]=[X_2, Y_2]=Z_1,  [X_2, Y_1]=[X_1, Y_2]=Z_2, \\
&[X_1, Z_1]=[X_2, Z_2]=A, [Y_1, Z_1]=[Y_2, Z_2]=B.
}
\end{rem}  
\subsection{Relation to previous work}
In \cite[Example 2.6 and Proposition 2.7]{CT95}, the authors give infinitely many examples of quadratically presented 3-step nilpotent Lie algebras. Moreover, they classify quadratically presented nilpotent Lie algebras whose first Betti number is at most 5. These Lie algebras arise from the fundamental groups of smooth projective (K\"ahler) manifolds. This condition {\bf (W)} corresponds to the case of 
\ali{
H^1(\fr{n}_{\bb{C}})&=H_{1,0}^1(\fr{n}_{\bb{C}})\o H_{0,1}^1(\fr{n}_{\bb{C}})
}
and
\ali{
H^2(\fr{n}_{\bb{C}})&=H_{2,0}^2(\fr{n}_{\bb{C}})\o H_{0,2}^2(\fr{n}_{\bb{C}})\o H_{1,1}^2(\fr{n}_{\bb{C}})
}
in \cref{W}.
Thus, this paper can be viewed as a quasi-projective analogue of \cite{CT95}.
The example of the $8$-dimensional nilpotent Lie algebra $\fr{n}$ in \cref{example} is $3$-step and not quadratically presented since the cup product 
\[
H^1(\fr{n}_{\bb{C}})\w H^1(\fr{n}_{\bb{C}})\rightarrow H^2(\fr{n}_{\bb{C}})
\]
is not surjective. The Lie algebra $\fr{n}$ gives a new candidate for a Lie algebra satisfying the mixed-Hodge-theoretic restrictions arising from smooth quasi-projective varieties. 

Moreover, none of the non-abelian lattices of rank at most seven appearing in this paper is a K\"ahler group. 
The lattices in the 5- and 7-dimensional Heisenberg groups are ruled out by the theorem of Carlson--Toledo (\cite{CT95}). 
On the other hand, for all the remaining cases, the corresponding Lie algebras are not quadratically presented. 
Hence these cases are also ruled out, since Carlson--Toledo proved that the Malcev Lie algebra of a nilpotent K\"ahler group is quadratically presented.
\subsection{Acknowledgements}
The author thanks his supervisor, Professor Hisashi Kasuya, Nagoya University, for his enormous support and for pointing out how to state the main theorems. He is also grateful to his supervisor, Professor Katsutoshi Yamanoi, The University of Osaka, for his valuable comments regarding the example in \cref{8dim}.
The author also benefited greatly from Rogov’s paper \cite{Rogov},  which was particularly helpful for \cref{Campana}. This work was supported by JST SPRING,  Grant Number JPMJSP2138.

The author is deeply grateful to the reviewer for their exceptionally careful and detailed comments, which have significantly improved the manuscript. The author also sincerely thanks the reviewer for their positive and encouraging remarks.
\section{Mixed Hodge theory on nilpotent fundamental groups of smooth quasi-projective varieties}
In this section, we recall the restrictions on the fundamental groups of smooth quasi-projective varieties established by Morgan in \cite{Morgan}. These restrictions are expressed in terms of mixed Hodge structures on the associated nilpotent Lie algebras and will be used throughout the paper. We begin by recalling the definition of a mixed Hodge structure. From this section, we denote the complexification of a vector space $\fr{n}$ by $\fr{n}_{\bb{C}}=\fr{n}\otimes\bb{C}$.
\subsection{Mixed Hodge structures}(\cite{Del71,Morgan,PeterSteen})
An  $\bb{R}$-\textit{mixed Hodge structure} $(V, W, F)$ is defined by the following data:
\begin{itemize}
\item $V$ is an $\bb{R}$-vector space.
\item $W=\{W_k(V)\}_{k\in\bb{Z}}$ is a finite increasing filtration on $V$.
\item $F=\{F^p(V_{\bb{C}})\}_{p\in\bb{Z}}$ is a finite decreasing filtration on the complexification $V_{\bb{C}}:=V\otimes\bb{C}$ of $V$ and defines a Hodge structure on 
\ali{
Gr_k^W(V_\bb{C}):=\frac{W_k(V)\otimes\bb{C}}{W_{k-1}(V)\otimes\bb{C}}
}
of weight $k$ for all $k\in\bb{Z}$. The decreasing filtration is given by
\ali{
F^p(Gr_k^W(V_\bb{C}))=\frac{F^p(V_{\bb{C}})\cap (W_k\otimes \bb{C})}{W_{k-1}\otimes\bb{C}}.
}
{\it i.e.} there exists a bigrading $Gr_k^W(V_{\bb{C}})=\bo_{p+q=k}\mathcal{H}_{p, q}$ such that
\ali{
\overline{\mathcal{H}_{p, q}}=\mathcal{H}_{q, p}\ \text{ for all }p, q\in\bb{Z}\text{ with }p+q=k.
}
\end{itemize}
Let $(V, W, F)$ be a mixed Hodge structure. By Deligne Splitting \cite{Del71,Del94}, a mixed Hodge structure on $V$ induces the bigrading on the complexification $V_{\bb{C}}=\bo_{p, q}V_{p, q}$ such that
\ali{
\overline{V_{p, q}}\equiv V_{q, p}\mod \bo_{s+t<p+q}V_{s, t}.
}
Here,  $V_{p, q}$ is given by
\ali{
V_{p, q}:=F^p(V_{\bb{C}})\cap W_{p+q}(V_{\bb{C}})\cap \left (\overline{F^q(V_{\bb{C}})}\cap W_{p+q}(V_{\bb{C}})+\sum_{i\geq 2}\overline{F^{q-i+1}(V_{\bb{C}})}\cap W_{p+q-i}(V_{\bb{C}})\right)}
and $W_k(V_{\bb{C}}):=W_k\otimes\bb{C}$. Moreover, the bigrading satisfies
\ali{
W_k(V_{\bb{C}})&=\bo_{p+q\leq k}V_{p, q}
}
and
\ali{
F^p(V_{\bb{C}})&=\bo_{s\geq p, t\in\bb{Z}}V_{s, t}
.}

We call a bigrading
\ali{
V_{\bb{C}}=\bo_{p,q\in\bb{Z}}V_{p,q}
}
satisfying
\ali{
\overline{V_{p, q}}\equiv V_{q, p}\mod \bo_{s+t<p+q}V_{s, t}
}
a Hodge-theoretic bigrading on $V_{\bb{C}}$. 
\subsection{Mixed Hodge structures on Lie algebras}
\begin{definition}
Let $\frak{n}$ be a real Lie algebra, and let
$
\frak{n}_{\bb{C}}:=\frak{n}\otimes\bb{C}.
$
A Hodge-theoretic bigrading
\[
\frak{n}_{\bb{C}}=\bigoplus_{p,q\in\bb{Z}}\frak{n}_{p,q}
\]
on the Lie algebra $\frak{n}_{\bb{C}}$ is said to be compatible with the Lie bracket if 
\begin{align}
[\frak{n}_{p,q},\frak{n}_{s,t}]
\subset
\frak{n}_{p+s,q+t}
\end{align}
for all $p,q,s,t\in\bb{Z}$.
In particular, the bracket is a morphism of mixed Hodge structures,
$\frak{n}$ is a Lie algebra in the category of mixed Hodge structures.
\end{definition}
Since the Lie bracket is a morphism of mixed Hodge structures, its dual and its multiplicative extension define a morphism of mixed Hodge structures on the Chevalley–Eilenberg complex. Hence, by the strictness of morphisms of mixed Hodge structures, the induced filtrations define a mixed Hodge structure on the Lie algebra cohomology (see \cite{Del71,Morgan}). Thus we have the following proposition.
\begin{prop}
The Hodge-theoretic bigrading on a Lie algebra induces the Hodge-theoretic bigrading on its Lie algebra cohomology.
\end{prop}
In particular, the induced Hodge-theoretic bigrading on the cohomology $H^*(\fr{n}_{\bb{C}})=\bo_{p,q}H_{p,q}^*$ is multiplicative. i.e. if $[x]\in H_{p,q}^j$ and $[y]\in H_{s,t}^k$, then $[x]\w [y]=[x\w y]\in H_{p+s,q+t}^{j+k}$.

Given a basis for each homogeneous component of a bigraded vector space,  for instance,  we denote the degree $(p, q)$ component,  whose basis consists of $X_1\dots X_n$ by $\ang{X_1, \dots , X_n}_{p, q}.$ 
\begin{eg}\label{Heisenberg}
Let $\fr{h}_3(\bb{R})$ be the $3$-dimensional Heisenberg Lie algebra. We can take a basis of the complexification $\fr{h}_3(\bb{R})\otimes\bb{C}=\fr{h}_3(\bb{C})=\ang{X_1, X_2, X_3}$ such that $[X_1, X_2]=X_3, [X_1, X_3]=[X_2, X_3]=0$ and $X_2=\overline{X_1}$. We define the bigrading on $\fr{h}_3(\bb{C})$ by
\ali{
\fr{h}_3(\bb{C})=\ang{X_1}_{-1, 0}\o\ang{X_2}_{0, -1}\o\ang{X_3}_{-1, -1}
.}
This is a Hodge-theoretic bigrading on $\fr{h}_3(\bb{R})$. The induced bigradings on the cohomologies are given by
\ali{
H^1(\fr{h}_3(\bb{C}))&=H_{1, 0}^1\o H_{0, 1}^1=\ang{[x_1]}_{1, 0}\o\ang{[x_2]}_{0, 1}, \\
H^2(\fr{h}_3(\bb{C}))&=H_{2, 1}^2\o H_{1, 2}^2=\ang{[x_1\w x_3]}_{2, 1}\o\ang{[x_2\w x_3]}_{1, 2}
}
and
\ali{
H^3(\fr{h}_3(\bb{C}))&=H_{2, 2}^3=\ang{[x_1\w x_2\w x_3]}_{2, 2}
}
where $x_1, x_2$ and $x_3$ form the dual basis of $X_1, X_2$ and $X_3$.
\end{eg}
\subsection{Bigradings of mixed Hodge structures and a restriction on the nilpotent fundamental groups of smooth quasi-projective varieties}
Let $M$ be a smooth quasi-projective variety. Assume that $\Gamma =\pi_1(M, x)$ is torsion-free nilpotent. By \cite{Malcev},  there uniquely exists a simply-connected nilpotent Lie group $N$ such that $\Gamma$ is a lattice (discrete and cocompact subgroup) in $N$. In \cite{Morgan},  Morgan showed a necessary condition of the Lie algebra $\fr{n}$ of $N$ in terms of the bigrading of a mixed Hodge structure on $\fr{n}$. More precisely,  Morgan showed the following theorem.
\begin{thm}\cite[Theorem {\bf (9.4)}]{Morgan}\label{W}
Let $M$ be a smooth quasi-projective variety. Let $N$ be a simply-connected real nilpotent Lie group. Assume that $\pi_1(M, x)$ is a lattice in $N$. For the Lie algebra $\fr{n}$ of $N$,  there exists a Hodge-theoretic bigrading $\fr{n}_{\bb{C}}=\bo_{p, q\leq 0, p+q\leq -1}\fr{n}_{p, q}$ such that for the induced bigrading $H^*(\fr{n}_{\bb{C}})=\bo_{p, q\geq 0, p+q\geq 1}H_{p, q}^*$,  we have 
\ali{
&{\bf (W)}:
\begin{array}{l}
H^1(\fr{n}_{\bb{C}})
 = H_{1,0}^1(\fr{n}_{\bb{C}}) \o H_{0,1}^1(\fr{n}_{\bb{C}}) \o H_{1,1}^1(\fr{n}_{\bb{C}}) \text{ and }\\
H^2(\fr{n}_{\bb{C}})
 = H_{2,0}^2(\fr{n}_{\bb{C}}) \o H_{1,1}^2(\fr{n}_{\bb{C}}) \o H_{0,2}^2(\fr{n}_{\bb{C}})
 \o H_{2,1}^2(\fr{n}_{\bb{C}}) \o H_{1,2}^2(\fr{n}_{\bb{C}}) \o H_{2,2}^2(\fr{n}_{\bb{C}}).
\end{array}
}
\end{thm}
We denote by $\mathcal{M}$ the class of nilpotent Lie algebras over $\bb{R}$ such that its complexification admit Hodge-theoretic bigradings satisfying condition {\bf (W)} of \cref{W}. The condition {\bf (W)} is derived from the bigradings of canonical mixed Hodge structures on the cohomologies of complex algebraic varieties (see \cite{Del71}). Consequently, we have the following corollary.
\begin{cor}
Let $\Gamma$ be a lattice in a simply-connected nilpotent Lie group $N$. If the complexification of the Lie algebra $\fr{n}_{\bb{C}}$ of N does not have any Hodge-theoretic bigrading satisfying the condition {\bf (W)}, then $\Gamma$ is not isomorphic to the fundamental group of any smooth quasi-projective variety.
\end{cor}
\begin{eg}\label{abelians}
Let $\bb{C}$ be the $1$-dimensional abelian Lie algebra. Take a non-zero element $X$. Then the bigrading $\bb{C}=\ang{X}_{-1, -1}$ is a Hodge-theoretic bigrading satisfying the condition {\bf (W)}. In the case of the $2$-dimensional abelian Lie algebra $\bb{C}^2$,  we may set $\bb{C}^2=\ang{X}_{-1, 0}\o\ang{\overline{X}}_{0, -1}$ where $X$ and $\overline{X}$ form a basis of $\bb{C}^2.$  
\end{eg}
\begin{eg}
Let $H_3(\bb{R})$ be the $3$-dimensional Heisenberg group and $\Gamma$ its lattice. Morgan showed that $M=\bb{R}\times \Gamma\backslash H_3(\bb{R})$ is diffeomorphic to a $\bb{C}^*$-bundle over a torus $T^2$ (see \cite{Morgan}). Thus $M$ is diffeomorphic to a smooth quasi-projective variety over $\bb{C}$ if a complex structure on $T^2$ is defined. Since $\pi_1(M,x)\cong\Gamma$, it follows that the complexification $\fr{h}_3(\bb{C})$ of the Lie algebra $\fr{h}_3(\bb{R})$ of $H_3(\bb{R})$ admits a Hodge-theoretic bigrading satisfying the condition {\bf (W)} in \cref{W}. The bigrading is given by \cref{Heisenberg}.
\end{eg}
\section{Some properties of Lie algebras in $\mathcal{M}$}
In this section, we show some properties of Lie algebras in $\mathcal{M}$. 
\begin{prop}\label{sum}
If $\fr{n}_1$ and $\fr{n}_2$ are in $\mathcal{M}$,  then $\fr{n}_1\o\fr{n}_2$ is in $\mathcal{M}$.
\end{prop}
\begin{proof}
Let ${\fr{n}_1}_{\bb{C}}=\bo_{p+q\leq -1,p,q\leq 0}(\fr{n}_1)_{p,q}$ and ${\fr{n}_2}_{\bb{C}}=\bo_{p+q\leq -1,p,q\leq 0}(\fr{n}_2)_{p,q}$ be bigradings on the complexifications ${\fr{n}_1}_{\bb{C}}=\fr{n}_1\otimes \bb{C}$ and ${\fr{n}_2}_{\bb{C}}=\fr{n}_2\otimes \bb{C}$. Define  
\ali{
(\fr{n}_1\o\fr{n}_2)_{p,q}:=(\fr{n}_1)_{p,q}\o (\fr{n}_2)_{p,q}.
}
Then the first and second cohomologies of $(\fr{n}_1\o \fr{n}_2)_{\bb{C}}=\bo_{p+q\leq -1,p,q\leq 0}(\fr{n}_1\o\fr{n}_2)_{p,q}$ are given by
\ali{
H^1((\fr{n}_1\o\fr{n}_2)_{\bb{C}})&=H^1({\fr{n}_1}_{\bb{C}})\o H^1({\fr{n}_2}_{\bb{C}})
}
and
\ali{
H^2((\fr{n}_1\o\fr{n}_2)_{\bb{C}})&=(H^1({\fr{n}_1}_{\bb{C}})\w H^1({\fr{n}_2}_{\bb{C}}) )\o H^2({\fr{n}_1}_{\bb{C}})\o H^2({\fr{n}_2}_{\bb{C}}).
}
Since the bigrading on $H^1({\fr{n}_1}_{\bb{C}})\w H^1({\fr{n}_2}_{\bb{C}})$ is given by
\ali{
H^1({\fr{n}_1}_{\bb{C}})\w H^1({\fr{n}_2}_{\bb{C}})&=H_{2,0}\o H_{1,1}\o H_{0,2}\o H_{2,1}\o H_{1,2}\o H_{2,2};\\
H_{2,0}&=H^1_{1,0}({\fr{n}_1}_{\bb{C}})\w H^1_{1,0}({\fr{n}_2}_{\bb{C}}),\\
H_{1,1}&=H^1_{1,0}({\fr{n}_1}_{\bb{C}})\w H^1_{0,1}({\fr{n}_2}_{\bb{C}})\o H^1_{1,0}({\fr{n}_2}_{\bb{C}})\w H^1_{0,1}({\fr{n}_1}_{\bb{C}}),\\
H_{0,2}&=H^1_{0,1}({\fr{n}_1}_{\bb{C}})\w H^1_{0,1}({\fr{n}_2}_{\bb{C}}),\\
H_{2,1}&=H^1_{1,1}({\fr{n}_1}_{\bb{C}})\w H^1_{1,0}({\fr{n}_2}_{\bb{C}})\o H^1_{1,1}({\fr{n}_2}_{\bb{C}})\w H^1_{1,0}({\fr{n}_1}_{\bb{C}}),\\
H_{1,2}&=H^1_{1,1}({\fr{n}_1}_{\bb{C}})\w H^1_{0,1}({\fr{n}_2}_{\bb{C}})\o H^1_{1,1}({\fr{n}_2}_{\bb{C}})\w H^1_{0,1}({\fr{n}_1}_{\bb{C}})
}
and
\ali{
H_{2,2}&=H^1_{1,1}({\fr{n}_1}_{\bb{C}})\w H^1_{1,1}({\fr{n}_2}_{\bb{C}}),\\
}
we have $\fr{n}_1\o\fr{n}_2\in \mathcal{M}$. 
\end{proof}
\begin{cor}
An abelian Lie algebra $\bb{R}^n$ is in $\mathcal{M}$.
\end{cor}
\begin{proof}
Since $\bb{R}^n$ is the direct sum of the $1$-dimensional abelian Lie algebra $\bb{R}^n=\bb{R}\o\cdots\o\bb{R}$, by \cref{abelians} and \cref{sum}, we see that $\bb{R}^n\in\mathcal{M}$. 
\end{proof}
The following proposition is useful for proving the main results in the next section.
\begin{prop}\label{useful}
Let $\fr{n}\in\mathcal{M}$ and $\fr{n}_{\bb{C}}=\bo_{p+q\leq -1,p,q\leq 0}\fr{n}_{p,q}$ be a bigrading on the complexification of $\fr{n}$ satisfying the condition in \cref{W}. Assume that there exists an element $0\neq [x]\in H_{p,q}^1(\fr{n}_{\bb{C}})$ and $z_1\in \fr{n}_{-s,-t}^*$ such that $dz_1=x\w y$ for some $y\in\fr{n}_{\bb{C}}^*$. Then we have
\ali{
(p+s,q+t)\in\{(2,0),(0,2),(1,1),(2,1),(1,2),(2,2)\}.
}
\end{prop}
\begin{proof}
We prove the above proposition by contradiction. Suppose that 
\ali{
(p+s,q+t)\notin\{(2,0),(0,2),(1,1),(2,1),(1,2),(2,2)\}.
}
Since 
\[
d(x\w z_1)=0, x\w z_1\in C^2_{2p+s,2q+t}
\]
and
\[
(2p+s,2q+t)\notin\{(2,0),(0,2),(1,1),(2,1),(1,2),(2,2)\},
\]
there exists $z_2\in\fr{n}_{-(2p+s),-(2q+t)}^*$ such that $dz_2=x\w z_1$. Inductively, there exists $z_n\in\fr{n}_{-(np+s),-(nq+t)}^*$ such that $dz_n=x\w z_{n-1}$ for $n\geq 2$. Since the elements $z_n$ belong to distinct bigraded components, $\{ z_n\}_{n\geq 1}$ are linearly independent. Thus $\fr{n}$ is infinite-dimensional. This contradicts that $\fr{n}$ is finite-dimensional vector space. Therefore 
\ali{
(p+s,q+t)\in\{(2,0),(0,2),(1,1),(2,1),(1,2),(2,2)\}.
}  
\end{proof}
\section{Main theorem: the case of $b_1(M)\leq 3$}\label{9} 
Let $M$ be a smooth quasi-projective variety. We denote the first Betti number of $M$ by $b_1(M)=\dim H^1(M, \bb{R})$. In this section,  we show the following theorem:
\begin{thm}\label{b1<=3}
Let $M$ be a smooth quasi-projective variety over $\bb{C}$. If $b_1(M)\leq 3$ and $\pi_1(M, x)$ is torsion-free nilpotent,  then $\pi_1(M, x)$ is abelian or $2$-step nilpotent. More precisely,  if $b_1(M)\leq 3$ and $\pi_1(M, x)$ is torsion-free nilpotent,  then $\pi_1(M, x)$ is isomorphic to one of the groups of the following  list.
\ali{
&\bullet\ \bb{Z} &&\text{ if }b_1(M)=1. \\ 
&\bullet\ \bb{Z}^2 \text{ or a lattice in }H_3(\bb{R}) &&\text{ if }b_1(M)=2.\\
&\bullet\ \bb{Z}^3\text{ or a lattice in }H_3(\bb{R})\times\bb{R} &&\text{ if }b_1(M)=3.
}
Here,  $H_3(\bb{R})$ is the $3$-dimensional Heisenberg group. In particular,  if $b_1(M)\leq 3$,  then $\pi_1(M, x)$ is abelian or $2$-step.
\end{thm}
Since $\Gamma=\pi_1(M,x)$ is torsion-free nilpotent, by Mal'cev's theorem in \cite{Malcev} and \cite{Raghu}, there exists a unique simply connected nilpotent Lie group $N$ such that $\Gamma$ is a lattice in $N$. Let $\frak n$ be the Lie algebra of $N$. Since we have
\[
b_1(M)=\dim H^1(M,\bb{R})=\dim H^1(\Gamma,\bb{R})=\dim(\Gamma^{ab}\otimes\bb{R}),
\]
and 
\[
\dim(\Gamma^{ab}\otimes\bb{R})=\dim H^1(\frak n),
\]
we obtain
\[
b_1(M)=\dim H^1(\fr{n}).
\]
Here, $\Gamma^{ab}$ denotes the abelianization of $\Gamma$ and $\fr{n}_{\bb{C}}$ the complexification. We next show the following results. \cref{b1<=3} is shown by the following result since $b_1(M)=\dim H^1(\fr{n})$. It is clear that $\Gamma\cong\bb{Z}$ if $b_1(M)=1$. 
\begin{prop}\label{casesH1=2}
Let $\fr{n}\in\mathcal{M}$. Assume that $\dim H^1(\fr{n})=2$. Then $\fr{n}$ is isomorphic to $\bb{R}^2\text{ or }\fr{h}_3(\bb{R})$.
\end{prop}
\begin{proof}
We show that $\fr{n}\cong \fr{h}_3(\bb{R})$ if $\fr{n}$ is non-abelian. We first show that $\dim H_{1,1}^1(\fr{n}_{\bb{C}})=0$ by contradiction. Assume that $\dim H_{1,1}^1(\fr{n}_{\bb{C}})\neq 0$. Then we have 
\[
\dim H_{1,1}^1(\fr{n}_{\bb{C}})=2
\]
since $\dim H_{1,0}^1(\fr{n}_{\bb{C}})=\dim H_{0,1}^1(\fr{n}_{\bb{C}})$. Let $H_{1,1}^1(\fr{n}_{\bb{C}})=\ang{[x_1],[x_2]}$. Since $\fr{n}$ is non-abelian, there exists $z\in\fr{n}_{-2,-2}^*$ such that $dz=x_1\w x_2$. This contradicts \cref{useful}. Thus 
\[
\dim H_{1,1}^1(\fr{n}_{\bb{C}})=0.
\] 
We next show $\fr{n}\cong\fr{h}_3(\bb{R})$. Let $H_{1,0}^1(\fr{n}_{\bb{C}})=\ang{[x]},H_{0,1}^1(\fr{n}_{\bb{C}})=\ang{[\overline{x}]}$. Since $\fr{n}$ is non-abelian, there exists $z\in\fr{n}_{-1,-1}^*$ such that $dz=x\w \overline{x}$. Assume that $\fr{n}_{-2,-1}\neq \0$. Then there exists $a\in\fr{n}_{-2,-1}^*$ such that $da=x\w z$. This contradicts \cref{useful}. Thus $\fr{n}_{-2,-1}=\0=\fr{n}_{-1,-2}$ by the symmetry. Therefore $\fr{n}_{\bb{C}}^*$ is given by
\ali{
\fr{n}_{\bb{C}}^*=\ang{x,\overline x,z}:dz=x\w \overline x.
}
Hence $\fr{n}\cong\fr{h}_3(\bb{R})$ by the classification of nilpotent Lie algebras in \cite{Nil}.
\end{proof}
\begin{prop}\label{casesH1=3}
Let $\fr{n}\in\mathcal{M}$. Assume that $\dim H^1(\fr{n})=3$. Then $\fr{n}$ is isomorphic to $\bb{R}^3\text{ or }\fr{h}_3(\bb{R})\o\bb{R}$.
\end{prop}
\begin{proof}
We show that $\fr{n}\cong \bb{R}\o\fr{h}_3(\bb{R})$ if $\fr{n}$ is non-abelian. We first show that $\dim H_{1,1}^1(\fr{n}_{\bb{C}})=1$ by contradiction. Assume that $\dim H_{1,1}^1(\fr{n}_{\bb{C}})\neq 1$. Then we have 
\[
\dim H_{1,1}^1(\fr{n}_{\bb{C}})=3
\]
since $\dim H_{1,0}^1(\fr{n}_{\bb{C}})=\dim H_{0,1}^1(\fr{n}_{\bb{C}})$. Let $H_{1,1}^1(\fr{n}_{\bb{C}})=\ang{[x_1],[x_2],[x_3]}$. Since $\fr{n}$ is non-abelian, there exists $z\in\fr{n}_{-2,-2}^*$ such that 
\[
dz=a_1x_1\w x_2+a_2x_1\w x_3+a_3x_2\w x_3
\]
for some $(a_1,a_2,a_3)\neq (0,0,0)$. We may assume that $a_1=1$. Then we have 
\ali{
dz=(x_1-a_3x_3)\w (x_2+a_2x_3).
}
Since $[x_1-a_3x_3]\neq 0$ in $H_{1,1}^1(\fr{n}_{\bb{C}})$, this contradicts \cref{useful}. Thus $\dim H_{1,1}^1(\fr{n}_{\bb{C}})=1$. 

We next show $\fr{n}\cong\bb{R}\o\fr{h}_3(\bb{R})$. Let $H_{1,0}^1(\fr{n}_{\bb{C}})=\ang{[x]},H_{0,1}^1(\fr{n}_{\bb{C}})=\ang{[\overline{x}]}$ and $H_{1,1}^1(\fr{n}_{\bb{C}})=\ang{[y]}$. If there exists $a\in\fr{n}_{-2,-1}^*$ such that $da=x\w w$ for some $w\in\fr{n}_{-1,-1}^*$, then this contradicts \cref{useful}. Hence $\fr{n}_{-2,-1}=\0 =\fr{n}_{-1,-2}$ by the symmetry. Thus there exists $z\in\fr{n}_{-1,-1}^*$ such that $dz=x\w \overline{x}$. Therefore $\fr{n}_{\bb{C}}^*$ is given by
\ali{
\fr{n}_{\bb{C}}^*=\ang{x,\overline x,y,z}:dz=x\w \overline x.
}
Hence $\fr{n}\cong\bb{R}\o\fr{h}_3(\bb{R})$ by the classification of nilpotent Lie algebras in \cite{Nil}.
\end{proof}
\section{Main theorem: The case of $rk(\pi_1(M, x))\leq 7$}
In this section,  we show the following results. Since non-Heisenberg examples first arise when $b_1(M)=4$, in \cref{rk<=7} we focus on the seven-dimensional indecomposable nilpotent Lie algebras with three-dimensional center satisfying the condition in \cref{W}. We then determine their complexifications in the case where the corresponding simply connected Lie groups contain $\pi_1(M,x)$ as a lattice. 
\begin{thm}\label{rk<=6}
Let $M$ be a smooth quasi-projective variety. If $\pi_1(M, x)$ is torsion-free nilpotent and $rk(\pi_1(M, x))\leq 6$,  then $\pi_1(M, x)$ is a lattice in one of the Lie groups of the following list:
\ali{
\ &\bullet\ \bb{R}&\text{ if }b_1(M)=1. \\
\ &\bullet\ \bb{R}^2\text{ or }H_3(\bb{R})&\text{ if }b_1(M)=2. \\
\ &\bullet\ \bb{R}^3\text{ or }\bb{R}\times H_3(\bb{R})&\text{ if }b_1(M)=3. \\
\ &\bullet\ \bb{R}^4,  \bb{R}^2\times H_3(\bb{R}),  H_3(\bb{R})\times H_3(\bb{R}), H_5(\bb{R})\text{ or }N_6(t)(t=-1,0)&\text{ if }b_1(M)=4. \\
\ &\bullet\ \bb{R}^5,  \bb{R}^3\times H_3(\bb{R})\text{ or }\bb{R}\times H_5(\bb{R})&\text{ if }b_1(M)=5. \\
\ &\bullet\ \bb{R}^6&\text{ if }b_1(M)=6.
}  
Here, $H_{2k+1}(\bb{R})$ is the $(2k+1)$-dimensional Heisenberg group and the Lie algebra $\fr{n}_6(t)$ of $N_6(t)$ is given by
\ali{
\fr{n}_6(t)=\ang{X_1,X_2,X_3,X_4,X_5,X_6}:[X_1,X_2]=[X_3,X_4]=X_5,[X_1,X_3]=X_6,[X_2,X_4]=tX_6.
}
\end{thm}
\begin{thm}\label{rk<=7}
Let $M$ be a smooth quasi-projective variety. Assume that $\pi_1(M, x)$ is torsion-free nilpotent and $rk(\pi_1(M, x))=7$. Let $N$ be the simply-connected nilpotent Lie group such that $\pi_1(M,x)$ is isomorphic to a lattice in $N$. Then $\pi_1(M, x)$ is abelian or $2$-step nilpotent. In particular, if $b_1(M)=4$, then the complexification of the Lie algebra $\fr{n}$ of $N$ is isomorphic to one of the following Lie algebras $\fr{n}_7^{142}$ or $\fr{n}_7^{143}$.
\ali{
\fr{n}_7^{142}=\ang{X_1, \dots, X_7}: [X_1, X_i]=X_{i-1}\text{ for i=3, 5, 7 }, [X_3, X_5]=X_4, [X_5, X_7]=X_2
}
and
\ali{
\fr{n}_7^{143}=\ang{X_1, \dots, X_7}: [X_1, X_i]=X_{i-1}\text{ for i=3, 5, 7 }, [X_3, X_5]=X_6, [X_5, X_7]=X_2
.}
\end{thm}
\begin{rem}\label{egdim7}
There exists a smooth quasi-projective variety $M$ over $\bb{C}$ such that $\pi_1(M,*)$ is isomorphic to a lattice in a simply-connected nilpotent Lie group $N$  whose Lie algebra is isomorphic to $\fr{n}_7^{142}$. We construct $M$ in the following way.
Let $H_1,H_2$ and $H_3$ be Hermitian forms on $\bb{C}^2$ given by
\ali{
H_1(z,w)&=\overline{z_1}w_1,\\
H_2(z,w)&=\overline{z_2}w_2,\\
H_3(z,w)&=\overline{z_1}w_2+\overline{z_2}w_1
}
for $z=(z_1,z_2),w=(w_1,w_2)\in\bb{C}^2$. Let $A=\bb{C}^2/\Lambda,\Lambda=(\bb{Z}+\sqrt{-1}\bb{Z})^2$ be an abelian variety. Then we have
\ali{
\Im H_1(\Lambda,\Lambda),\Im H_2(\Lambda,\Lambda),\Im H_3(\Lambda,\Lambda)\subset \bb{Z}.
}
Thus, by \cite{BL}, there exist holomorphic line bundles $L_1,L_2$ and $L_3$ over $A$ such that
\ali{
c_1(L_1)=[\Im H_1],c_1(L_2)=[\Im H_2],c_1(L_3)=[\Im H_3].
}
Moreover, by \cite{GAGA}, $L_1,L_2$ and $L_3$ are smooth quasi-projective varieties since $A$ is a smooth projective variety. Let $L_i^{\times}$ denote the complement of the zero-section. Then we have that $L_i^{\times}\rightarrow A$ is a principal $\bb{C}^*$-bundle for $i=1,2,3$. Let $M$ be the pull-back of 
\ali{
L_1^{\times}\times L_2^{\times}\times L_3^{\times}\rightarrow A\times A\times A
}
by $\Delta:A\rightarrow A\times A\times A$. Since $M$ is homotopy equivalent to a  principal $T^3$-bundle over a torus $T^4$ and $\pi_1(T^n,*)\cong\bb{Z}^n$, we have a central extension
\ali{
1\rightarrow \bb{Z}^3\rightarrow \pi_1(M,*)\rightarrow\bb{Z}^4\rightarrow 1.
}
Let $\Lambda:=\ang{\alpha_1,\alpha_2,\beta_1,\beta_2}$ where $\alpha_1=(1,0),\beta_1=(\sqrt{-1},0),\alpha_2=(0,1)$ and $\beta_2=(0,\sqrt{-1})$. Then we have the following group presentation: $\pi_1(M,*)$ is generated by 
\ali{
\tilde \alpha_1,\tilde\alpha_2,\tilde\beta_1,\tilde\beta_2,\gamma_1,\gamma_2\text{ and }\gamma_3.
} 
Since $\Im H_1(\alpha_1,\beta_1)=\Im H_2(\alpha_2,\beta_2)=\Im H_3(\alpha_1,\beta_2)=\Im H_3(\alpha_2,\beta_1)=1$, the relations are given by
\ali{
[\tilde\alpha_1,\tilde\alpha_2]=[\tilde\beta_1,\tilde\beta_2]=[\tilde\alpha_i,\gamma_j]=[\tilde\beta_k,\gamma_l]=e
}
for $i,k=1,2$ and $j,l=1,2,3$,
\ali{
[\tilde\alpha_1,\tilde\beta_1]&=\gamma_1^{\Im H_1(\alpha_1,\beta_1)}=\gamma_1,\\
[\tilde\alpha_2,\tilde\beta_2]&=\gamma_2^{\Im H_2(\alpha_2,\beta_2)}=\gamma_2,\\
[\tilde\alpha_1,\tilde\beta_2]&=\gamma_3^{\Im H_3(\alpha_1,\beta_2)}=\gamma_3
}
and
\ali{
[\tilde\alpha_2,\tilde\beta_1]&=\gamma_3^{\Im H_3(\alpha_2,\beta_1)}=\gamma_3.
}
Then $\pi_1(M,*)$ is isomorphic to a lattice in a Lie group $N$ such that the complexification $\fr{n}_{\bb{C}}$ of the Lie algebra $\fr{n}$ is given by
\ali{
\fr{n}_{\bb{C}}=\ang{e_1,e_2,e_3,e_4,e_5,e_6,e_7}:[e_1,e_2]=e_5,[e_3,e_4]=e_6,[e_1,e_4]=[e_2,e_3]=e_7.
}
Therefore $\fr{n}_{\bb{C}}\cong\fr{n}_7^{142}$. A similar construction, using the Hermitian forms
\ali{
H_1(z,w)&=\overline{z_1}w_1,\\
H_2(z,w)&=\overline{z_2}w_1+\overline{z_1}w_2
}
and
\ali{
H_3(z,w)&=-\sqrt{-1}\overline{z_2}w_1+\sqrt{-1}\overline{z_1}w_2,
}
yields a smooth quasi-projective variety whose fundamental group is a lattice in a simply connected nilpotent Lie group with complexified Lie algebra $\fr{n}_7^{143}$.

\end{rem}
\begin{prob}
Let $N$ be a simply-connected nilpotent Lie group such that the Lie algebra $\fr{n}$ of $N$ is isomorphic to $\fr{n}_7^{142}$ or $\fr{n}_7^{143}$. Let $\Gamma$ be any lattice in $N$. Then does there exist a smooth quasi-projective variety $M$ such that $\pi_1(M,x)$ is isomorphic to $\Gamma$?
\end{prob}
\section{The proof of \cref{rk<=6}}
\subsection{Some lemmas for the proof of \cref{rk<=6}}
In order to show \cref{rk<=6}, we show the following lemmas.
\begin{lem}\label{H11not4}
Let $\fr{n}\in\mathcal{M}$. Assume that $\dim\fr{n}=6$ and $\dim H^1(\fr{n})=4$. Then $\dim H_{1,1}^1(\fr{n}_{\bb{C}})\neq 4$.
\end{lem}
\begin{proof}
We show this lemma by contradiction. Assume that $\dim H_{1,1}^1(\fr{n}_{\bb{C}})=4$. Let $H_{1,1}^1(\fr{n}_{\bb{C}})=\ang{[x_1],[x_2],[x_3],[x_4]}$. Since $\dim\fr{n}=6$, we have 
\[
\dim\fr{n}_{-2,-2}^*=\dim\fr{n}_{-3,-3}=1\text{ or }\dim\fr{n}_{-2,-2}=2.
\]
If $\dim\fr{n}_{-2,-2}=\dim\fr{n}_{-3,-3}=1$, then there exist $y\in\fr{n}_{-2,-2}^*$ and $z\in\fr{n}_{-3,-3}^*$ such that 
\[
dz=(\sum_{i=1}^4a_ix_i)\w y
\]
for some $a_1,a_2,a_3,a_4\in\bb{C}$. This contradicts \cref{useful}. Thus $\dim\fr{n}_{-2,-2}=2$. For the differential $d:\fr{n}_{-1,-1}^*\w \fr{n}_{-2,-2}^*\rightarrow\bw^3\fr{n}_{-1,-1}^*$, we have
\ali{
&\dim\Ker (d:\fr{n}_{-1,-1}^*\w \fr{n}_{-2,-2}^*\rightarrow\bw^3\fr{n}_{-1,-1}^*)\\
&=\dim (\fr{n}_{-1,-1}^*\w \fr{n}_{-2,-2}^*)-\dim\Im (d:\fr{n}_{-1,-1}^*\w \fr{n}_{-2,-2}^*\rightarrow\bw^3\fr{n}_{-1,-1}^*)\\
&\geq 8-4=4.
}
Thus $H_{3,3}^2(\fr{n}_{\bb{C}})\neq \0$. This contradicts the condition in $\mathcal{M}$. Therefore $\dim H_{1,1}^1(\fr{n}_{\bb{C}})\neq 4$.
\end{proof}
\begin{lem}\label{H11not2}
Let $\fr{n}\in\mathcal{M}$. Assume that $\dim\fr{n}=6$ and $\dim H^1(\fr{n})=4$. Then $\dim H_{1,1}^1(\fr{n}_{\bb{C}})\neq 2$.
\end{lem}
\begin{proof}
We show this lemma by contradiction. Assume that $\dim H_{1,1}^1(\fr{n}_{\bb{C}})=2$. Let $H_{1,0}^1(\fr{n}_{\bb{C}})=\ang{[x]},H_{0,1}^1(\fr{n}_{\bb{C}})=\ang{[\overline{x}]}$ and $H_{1,1}^1(\fr{n}_{\bb{C}})=\ang{[y_1],[y_2]}$. If $\fr{n}_{-2,-1}\neq \0$, then there exists $a\in\fr{n}_{-2,-1}^*$ such that $da=x\w z$ for some $z\in\fr{n}_{-1,-1}^*$. This contradicts \cref{useful}. Thus 
\[
\fr{n}_{-2,-1}=\fr{n}_{-1,-2}=\0.
\]
Hence there exists $z\in\fr{n}_{-1,-1}^*$ such that $dz=x\w \overline{x}$. Since $\dim\fr{n}=6$, we have $\fr{n}_{-2,-2}\neq \0$. Thus there exists $w\in\fr{n}_{-2,-2}^*$ such that 
\[
dw=(a_1y_1+a_2y_2)\w z
\]
for some $a_1,a_2\in\bb{C}$. This contradicts \cref{useful}. Thus $\fr{n}_{-2,-2}=\0$. Hence $\dim \fr{n}=5\neq 6$. Therefore $\dim H_{1,1}^1(\fr{n}_{\bb{C}})\neq 2$.
\end{proof}
\begin{lem}\label{H1=4}
Let $\fr{n}\in\mathcal{M}$. Assume that $\dim\fr{n}\leq 6$, $\dim H^1(\fr{n})=4$ and $\fr{n}\not\cong \fr{n}_1\o\fr{n}_2$ for any nilpotent Lie algebras $\fr{n}_1$ and $\fr{n}_2$. Then $\fr{n}$ is isomorphic to $\fr{h}_5(\bb{R})$ or $\fr{n}_6(t)(t=0,-1)$ where $\fr{n}_6(t)$ is given by
\ali{
\fr{n}_6(t)=\ang{X_1,X_2,X_3,X_4,X_5,X_6}:[X_1,X_2]=[X_3,X_4]=X_5,[X_1,X_3]=X_6,[X_2,X_4]=tX_6.
}
\end{lem}
\begin{proof}
By the classification of nilpotent Lie algebras in \cite{Nil}, $\fr{n}$ is isomorphic to $\fr{h}_5(\bb{R})$ if $\dim \fr{n}=5$.
We next show that $\fr{n}\cong\fr{n}_6(t)$ for some $t=0,-1$ if $\dim\fr{n}=6$. By \cref{H11not4} and \cref{H11not2}, we have 
\[
\dim H_{1,0}^1(\fr{n}_{\bb{C}})=\dim H_{0,1}^1(\fr{n}_{\bb{C}})=2.
\]
Let $H_{1,0}^1(\fr{n}_{\bb{C}})=\ang{[x_1],[x_2]}$. If $\fr{n}_{-2,0}\neq \0$, then there exists $x\in\fr{n}_{-2,0}^*$ such that $dx=x_1\w x_2$. This contradicts \cref{useful}. Thus 
\[
\fr{n}_{-2,0}=\fr{n}_{0,-2}=\0.
\]
We show that $\dim\fr{n}_{-1,-1}=2$. Assume that $\fr{n}_{-1,-1}=1$. Since $\dim \fr{n}=6$, we have 
\[
\fr{n}_{-2,-1},\fr{n}_{-1,-2}\neq \0
\]
and hence $\dim\fr{n}\geq 7$. This contradicts that $\dim\fr{n}=6$. Thus $\dim\fr{n}_{-1,-1}=2$. Hence $\fr{n}$ is a $6$-dimensional $2$-step nilpotent Lie algebra with $2$-dimensional center such that $\fr{n}\not\cong \fr{n}_1\o\fr{n}_2$ for any nilpotent Lie algebras $\fr{n}_1$ and $\fr{n}_2$. By the classification of nilpotent Lie algebras in \cite{Nil}, $\fr{n}$ is isomorphic to $\fr{n}_6(t)$ for some $t=0,-1$.
\end{proof}
\subsection{The proof of \cref{rk<=6}}
In this part, we show \cref{rk<=6}.
\begin{proof}
Let $N$ be a simply-connected nilpotent Lie group such that $\pi_1(M,x)$ is a lattice in $N$. Let $\fr{n}$ be the Lie algebra of $N$. Then $\fr{n}\in\mathcal{M}$. Since $b_1(M)=\dim H^1(\fr{n})$, it suffices to show that $\fr{n}$ is isomorphic to one of the following list.
\ali{
&\bullet\ \bb{R}&\text{ if }\dim H^1(\fr{n})=1,\\
&\bullet\ \bb{R}^2\text{ or }\fr{h}_3(\bb{R})&\text{ if }\dim H^1(\fr{n})=2,\\
&\bullet\ \bb{R}^3\text{ or }\bb{R}\o\fr{h}_3(\bb{R})&\text{ if }\dim H^1(\fr{n})=3,\\
&\bullet\ \bb{R}^4,\bb{R}^2\o\fr{h}_3(\bb{R}),\fr{h}_3(\bb{R})\o\fr{h}_3(\bb{R}),\fr{h}_5(\bb{R})\text{ or }\fr{n}_6(t)(t=0, -1)&\text{ if }\dim H^1(\fr{n})=4,\\
&\bullet\ \bb{R}^5,\bb{R}^3\o\fr{h}_3(\bb{R})\text{ or }\bb{R}\o\fr{h}_5(\bb{R})&\text{ if }\dim H^1(\fr{n})=5,\\
&\bullet\ \bb{R}^6&\text{ if }\dim H^1(\fr{n})=6
}
By \cref{b1<=3}, the case of $\dim H^1(\fr{n})\leq 3$ and $\dim H^1(\fr{n})=6$ are clear. By \cref{H1=4}, the case of $\dim H^1(\fr{n})=4$ holds. Moreover, by the classification of $5$ and $6$-dimensional nilpotent Lie algebras in \cite{Nil}, we have 
\[
\fr{n}\cong \bb{R}^5,\bb{R}^3\o\fr{h}_3(\bb{R})\text{ or }\bb{R}\o\fr{h}_5(\bb{R})
\]
if $\dim H^1(\fr{n})=5$. Therefore \cref{rk<=6} holds.
\end{proof}
\section{The proof of \cref{rk<=7}}
\subsection{Lemmas for the proof of \cref{rk<=7}}
In order to show \cref{rk<=7}, we first show the following lemmas.
\begin{lem}\label{H11neq4}
Let $\fr{n}\in\mathcal{M}$. Assume that $\dim\fr{n}=7$ and $\dim H^1(\fr{n})=4.$ Then we have $\dim H_{1,1}^1(\fr{n}_{\bb{C}})\neq 4$.
\end{lem}
\begin{proof}
We show this lemma by contradiction. Assume that $\dim H_{1,1}^1(\fr{n}_{\bb{C}})=4$. Since $\dim\fr{n}=7$, we have 
\[
(\dim \fr{n}_{-2,-2},\dim \fr{n}_{-3,-3},\dim \fr{n}_{-4,-4})=(1,1,1),(1,2,0),(2,1,0),(2,0,1)\text{ or }(3,0,0).
\]
If $\dim\fr{n}_{-2,-2}\geq 2$, then we have 
\ali{
&\dim\Ker (d:\fr{n}_{-1,-1}^*\w \fr{n}_{-2,-2}^*\rightarrow\bw^3\fr{n}_{-1,-1}^*)\\
&=\dim (\fr{n}_{-1,-1}^*\w \fr{n}_{-2,-2}^*)-\dim\Im (d:\fr{n}_{-1,-1}^*\w \fr{n}_{-2,-2}^*\rightarrow\bw^3\fr{n}_{-1,-1}^*)\\
&\geq 8-4=4.
} 
Thus $H_{3,3}^2(\fr{n}_{\bb{C}})\neq \0$. This contradicts the condition {\bf (W)} in \cref{W}. Thus $\dim \fr{n}_{-2,-2}=1$ and hence 
\ali{
(\dim \fr{n}_{-2,-2},\dim \fr{n}_{-3,-3},\dim \fr{n}_{-4,-4})=(1,1,1)\text{ or }(1,2,0).
}
Let $\fr{n}_{-2,-2}^*=\ang{y}$. In both cases, there exists $z\in\fr{n}_{-3,-3}^*$ such that $dz=x\w y$ for some $x\in\fr{n}_{-1,-1}^*$. Since $[x]\neq 0$ in $H_{1,1}^1(\fr{n}_{\bb{C}})$, this contradicts \cref{useful}. Thus we have
\ali{
(\dim \fr{n}_{-2,-2},\dim \fr{n}_{-3,-3},\dim \fr{n}_{-4,-4})\neq (1,1,1),(1,2,0).
}
Hence we have $\dim H_{1,1}^1(\fr{n}_{\bb{C}})\neq 4$.
\end{proof}
\begin{lem}\label{H11neq2}
Let $\fr{n}\in\mathcal{M}$. Assume that $\dim\fr{n}=7$ and $\dim H^1(\fr{n})=4.$ Then we have $\dim H_{1,1}^1(\fr{n}_{\bb{C}})\neq 2$.
\end{lem}
\begin{proof}
We prove this lemma by contradiction. Assume that $\dim H_{1,1}^1(\fr{n}_{\bb{C}})= 2$. Let $H^1_{1,0}(\fr{n}_{\bb{C}})=\ang{[x]},H^1_{0,1}(\fr{n}_{\bb{C}})=\ang{[\overline x]}$ and $H^1_{1,1}(\fr{n}_{\bb{C}})=\ang{[y_1],[y_2]}$.
We first show that $\fr{n}_{-2,-1}=\fr{n}_{-1,-2}=\0$. It suffices to show that $\fr{n}_{-2,-1}=\0$ by the symmetry. If $\fr{n}_{-2,-1}\neq \0$, then there exists $a\in\fr{n}_{-2,-1}^*$ such that $da=x\w w$ for some $w\in \fr{n}_{-1,-1}^*$. This contradicts \cref{useful}. Thus 
\[
\fr{n}_{-2,-1}=\0
\]
and hence 
$\fr{n}_{-1,-2}=\0$.

Second, we show that $\fr{n}_{-2,-2}=\0$. Assume that there does not exist $z\in\fr{n}_{-1,-1}^*$ such that $dz=x\w \overline x$. Since $\fr{n}$ is non-abelian, there exists $y\in\fr{n}_{-2,-2}^*$ such that $dy=y_1\w y_2$. This contradicts \cref{useful}. Thus there exists $z\in\fr{n}^*_{-1,-1}$ such that $dz=x\w \overline x$. However, if there exists $y\in\fr{n}_{-2,-2}^*$ such that 
\[
dy=(a_1y_1+a_2y_2)\w z
\]
for some $a_1,a_2\in\bb{C}$, then this contradicts \cref{useful}. Hence we have $\dim \fr{n}=5\neq 7$. Therefore $\dim H_{1,1}^1(\fr{n}_{\bb{C}})\neq 2$.
\end{proof}
\begin{lem}\label{2step}
Let $\fr{n}\in\mathcal{M}$. Assume that $\dim\fr{n}=7$ and $\dim H^1(\fr{n})=4$. Then $\fr{n}$ is $2$-step nilpotent.
\end{lem}
\begin{proof}
By \cref{H11neq4} and \cref{H11neq2}, we have 
\[
\dim H_{1,0}^1(\fr{n}_{\bb{C}})=\dim H_{0,1}^1(\fr{n}_{\bb{C}})=2.
\]
In the same way as the proof of \cref{H1=4}, we have $\fr{n}_{-2,0}=\fr{n}_{0,-2}=\0$. Thus it suffices to show that $\dim \fr{n}_{-1,-1}=3$. We show this by contradiction. 
We first assume that $\dim \fr{n}_{-1,-1}=1$. Let $\fr{n}_{-1,-1}^*=\ang{z}$. Since $\dim \fr{n}=7$, there exists $a\in\fr{n}_{-2,-1}^*$ such that $da=x\w z$ for some $x\in \fr{n}_{-1,0}^*$. This contradicts \cref{useful}. Thus 
\[
\dim \fr{n}_{-1,-1}\neq 1. 
\]
Second, we assume that $\dim \fr{n}_{-1,-1}=2$. Assume that $\fr{n}_{-2,-1}\neq \0$. Then there exists $a\in\fr{n}_{-2,-1}^*$ such that $da\neq 0$. Thus $\fr{n}_{-1,-2}\neq \0$ by the symmetry. Hence $\dim\fr{n}\geq 8$. This contradicts that $\dim\fr{n}=7$. Thus 
\[
\fr{n}_{-2,-1}=\fr{n}_{-1,-2}=\0.
\]
Let $\fr{n}_{-1,-1}^*=\ang{z_1,z_2}$. Then there exists $z\in\fr{n}_{-2,-2}^*$ such that $dz=z_1\w z_2$. Since 
\[
d^2z=dz_1\w z_2-z_1\w dz_2=0
\]
and $z_1$ and $z_2$ are linearly independent, we have $dz_1=0=dz_2$. This contradicts $\dim H^1(\fr{n}_{\bb{C}})=4$. Thus $\dim \fr{n}_{-1,-1}\neq 2$. Therefore $\dim \fr{n}_{-1,-1}=3$ and hence $\fr{n}$ is $2$-step nilpotent.
\end{proof}
\subsection{The proof of \cref{rk<=7}}
By \cref{2step}, $\fr{n}$ is $2$-step nilpotent and $\dim H_{1,1}^1(\fr{n}_{\bb{C}})=0$. 
Let $\fr{n}_{-1, 0}=\ang{X_1, X_2}$ and $\fr{n}_{0, -1}=\ang{\overline{X_1}, \overline{X_2}}$. By the classification of nilpotent Lie algebras of dimension up to $7$ over $\bb{C}$, we obtain the following.
\begin{claim}\label{rk7b=4}
$\fr{n}_{\bb{C}}$ is isomorphic to one of the following Lie algebras:
\ali{
\fr{n}_7^{142}=\ang{X_1, \dots, X_7}: [X_1, X_i]=X_{i-1}\text{ for i=3, 5, 7 }, [X_3, X_5]=X_4, [X_5, X_7]=X_2.
}
\ali{
\fr{n}_7^{143}=\ang{X_1, \dots, X_7}: [X_1, X_i]=X_{i-1}\text{ for i=3, 5, 7 }, [X_3, X_5]=X_6, [X_5, X_7]=X_2
.}
\end{claim}
\begin{proof}
We may assume that $Z_1:=[X_1, \overline{X_1}], Z_2:=[X_1, \overline{X_2}]$ and $Z_3:=[X_2, \overline{X_1}]$ form a basis of $\fr{n}_{-1, -1}$. Let $[X_2, \overline{X_2}]=aZ_1+bZ_2+cZ_3$. Note that
\ali{\overline{Z_1}&=\overline{[X_1, \overline{X_1}]} =[\overline{X_1}, X_1]=-Z_1, \\
\overline{Z_2}&=\overline{[X_1, \overline{X_2}]}=[\overline{X_1}, X_2]=-Z_3
}
and
\ali{
\overline{Z_3}&=-Z_2.
}
Let $e_1:=X_1, e_2:=\overline{X_1}, e_3:=\overline{X_2}, e_4:=X_2, e_5:=Z_1, e_6:=Z_2$ and $ e_7:=Z_3$. Then we have the structure constants $\{ C_{ij}^k\}_{i<j}$ such that
\ali{
C_{ij}^k=
\begin{cases}
1&(i, j, k)=(1, 2, 5), (1, 3, 6), (4, 2, 7)\\
a&(i, j, k)=(4, 3, 5)\\
b&(i, j, k)=(4, 3, 6)\\
c&(i, j, k)=(4, 3, 7)\\
0&\text{ otherwise}.
\end{cases}
}
By the classification of complex nilpotent Lie algebras (\cite{Nil}) and since the dimension of the center of $\fr{n}_{\bb{C}}$ is $3$,  we have
\ali{
\fr{n}_{\bb{C}}\cong \fr{n}_7^{142}\text{ or }\fr{n}_7^{143}\text{ or }\fr{n}_7^{144}\text{ or }\fr{n}_7^{145}.
}
By direct computation,  we have
\ali{
(\dim H^2(\fr{n}_7^i), \dim H^3(\fr{n}_7^i))
=
\begin{cases}
(11, 14)&i=142\\
(11, 16)&i=143\\
(11, 17)&i=144\\
(12, 18)&i=145.
\end{cases}
} 
Since
\ali{
(\dim H^2(\fr{n}_{\bb{C}}), \dim H^3(\fr{n}_{\bb{C}}))
=
\begin{cases}
(11, 14)&a\neq 0\text{ or }b\neq 0\\
(11, 16)&a=b=0, \\
\end{cases}
}
we have
\ali{
\fr{n}_{\bb{C}}\cong \fr{n}_7^{142}\text{ or }\fr{n}_7^{143}
.}
\end{proof}  
\section{The case of $rk(\pi_1(M, x))=8$}
Let $M$ be a smooth quasi-projective algebraic variety. Assume that $\pi_1(M, x)$ is torsion-free and $3$-step nilpotent. Let $N$ be the simply-connected nilpotent Lie group such that $\pi_1(M, x)$ is a lattice in $N$. By the main results,  we have $rk(\pi_1(M, x))=\dim N=\dim \fr{n}\geq 8$ where $\fr{n}$ is the Lie algebra of $N$. In this section,  we consider the case of $rk(\pi_1(M, x))=8$. We give an example of a $3$-step nilpotent Lie group whose lattice may or may not be isomorphic to the fundamental group of a smooth quasi-projective variety.  
\begin{prop}\label{8dim}
Let $\fr{n}=\ang{X_1, X_2, Y_1, Y_2, Z_1, Z_2, A, B}$ such that 
\ali{
[X_1, Y_1]=[X_2, Y_2]=Z_1,  [X_2, Y_1]=[X_1, Y_2]=Z_2, \\
[X_1, Z_1]=[X_2, Z_2]=A, [Y_1, Z_1]=[Y_2, Z_2]=B
.}
Then $\fr{n}$ is $3$-step and $\fr{n}\in \mathcal{M}$. 
In particular, the mixed Hodge structure constructed in the proof is defined over $\mathbb{Q}$.
\end{prop}
\begin{proof}
Since
\[
C^1\fr{n}=\fr{n},C^2\fr{n}=\ang{Z_1,Z_2,A,B},C^3\fr{n}=\ang{A,B}\text{ and }C^4\fr{n}=\0,
\]
$\fr{n}$ is $3$-step. We show that $\fr{n}\in\mathcal{M}.$ We define the Hodge-theoretic bigrading on $\fr{n}_{\bb{C}}$ to be 
\ali{
\fr{n}_{\bb{C}}&=\ang{A_1, A_2}_{-1, 0}\o\ang{\overline{A_1}, \overline{A_2}}_{0, -1}\o\ang{B_1, B_2}_{-1, -1}\o\ang{C_1}_{-2, -1}\o\ang{\overline{C_1}}_{-1, -2}
}
where $A_1:=X_1+iY_1, A_2:=X_2+iY_2, B_1:=-2iZ_1, B_2:=-2iZ_2$ and $C_1:=-2i(A+iB)$. Thus the bigrading defines a mixed Hodge structure on $\fr{n}$.
Let $a_1, a_1, \overline{a_1}, \overline{a_2}, b_1, b_2, c_1$ and $\overline{c_1}$ be the dual basis. Since 
\ali{
[A_1, \overline{A_1}]=B_1, [A_2, \overline{A_2}]=B_2, \\
[A_1, B_1]=[A_2, B_2]=C_1, \\
[\overline{A_1}, \overline{B_1}]=[\overline{A_2}, \overline{B_2}]=\overline{C_1}, 
}
and $[X, Y]=0\text{ for other }X, Y\in \{A_1, \dots,  \overline{C_1}\}, $ we have 
\ali{
da_1&=da_2=d\overline{a_1}=d\overline{a_2}=0, \\
db_1&=-a_1\w \overline{a_1}, db_2=-a_2\w \overline{a_2}
}
and
\ali{
dc_1&=-a_1\w b_1-a_2\w b_2, d\overline{c_1}=\overline{a_1}\w b_1+\overline{a_2}\w b_2.
}
Hence the induced bigradings on the first and second cohomologies are given by
\ali{
H^1(\fr{n})=\ang{[a_1], [a_2]}_{1, 0}\o\ang{[\overline{a_1}], [\overline{a_2}]}_{0, 1}
}
and
\ali{
H^2(\fr{n})&=\ang{[a_1\w a_2]}_{2, 0}\o\ang{[\overline{a_1}\w \overline{a_2}]}_{0, 2}\o\ang{[a_1\w \overline{a_2}], [\overline{a_1}\w a_2]}_{1, 1}\\
&\o\ang{[a_1\w b_1]}_{2, 1}\o\ang{[\overline{a_1}\w b_1]}_{1, 2}.
}
Therefore $\fr{n}\in\mathcal{M}.$ Moreover, a mixed Hodge structure on $\fr{n}$ is given by the following increasing filtration $W_{\bullet}$ on $\fr{n}$ and decreasing filtration $F^{\bullet}$ on $\fr{n}_{\bb{C}}$:
\ali{
W_{-1}(\fr{n})&=\fr{n},\ \\
W_{-2}(\fr{n})&=\ang{Z_1,Z_2,A,B},\ \\
W_{-3}(\fr{n})&=\ang{A,B},\\
W_{-i}(\fr{n})&=\0 \text{ for }i\geq 4.\\
F^{-2}(\fr{n}_{\bb{C}})&=\fr{n}_{\bb{C}},\ \\
F^{-1}(\fr{n}_{\bb{C}})&=\ang{A_1,A_2,\overline{A_1},\overline{A_2},B_1,B_2,\overline{C_1}},\ \\
F^{0}(\fr{n}_{\bb{C}})&=\ang{\overline{A_1},\overline{A_2},\overline{C_1}}
}
and
\ali{
F^i(\fr{n}_{\bb{C}})&=\0\text{ for }i\geq 1.
}
\end{proof}
\begin{rem}
Let $\fr{n}=\ang{X_1, X_2, Y_1, Y_2, Z_1, Z_2, A, B}$ be the Lie algebra in \cref{8dim}. Then $\fr{n}$ is isomorphic to the Lie subalgebra $\overline{\fr{n}}$ of $\fr{gl}_9(\bb{R})$ given by
\ali{
\overline{\fr{n}}:=\left \{\begin{pmatrix}
0&0&y_2&y_1&-z_2&-z_1&0&0&3b\\
0&0&x_2&x_1&0&0&-z_2&-z_1&3a\\
0&0&0&0&x_1&x_2&-y_1&-y_2&2z_2\\
0&0&0&0&x_2&x_1&-y_2&-y_1&2z_1\\
0&0&0&0&0&0&0&0&y_2\\
0&0&0&0&0&0&0&0&y_1\\
0&0&0&0&0&0&0&0&x_2\\
0&0&0&0&0&0&0&0&x_1\\
0&0&0&0&0&0&0&0&0\\
\end{pmatrix}
;
x_1, x_2, y_1, y_2, z_1, z_2, a, b\in\bb{R}\right \}
, }
which can be proved by constructing the isomorphism $\phi :\fr{n}\rightarrow \overline{\fr{n}}$ such that
\ali{
\phi (X_1)&:=e_{2, 4}+e_{3, 5}+e_{4, 6}+e_{8, 9},\\
\phi (X_2)&:=e_{2, 3}+e_{3, 6}+e_{4, 5}+e_{7, 9},\\
\phi (Y_1)&:=e_{1, 4}-e_{3, 7}-e_{4, 8}+e_{6, 9},\\
\phi (Y_2)&:=e_{1, 3}-e_{3, 8}-e_{4, 7}+e_{5, 9},\\
\phi (Z_1)&:=-e_{1, 6}-e_{2, 8}+2e_{4, 9},\\
\phi (Z_2)&:=-e_{1, 5}-e_{2, 7}+2e_{3, 9},\\
\phi (A)&:=3e_{2, 9}
}
and
\ali{
\phi (B)&:=3e_{1, 9}.
}
Furthermore, every element of the corresponding simply connected nilpotent Lie group \(G(\mathbb{R})\) is of the form
\ali{
\begin{pmatrix}
1&0&y_2&y_1&
-z_2+\frac{x_1y_2+x_2y_1}{2}&
-z_1+\frac{x_1y_1+x_2y_2}{2}&
-y_1y_2&
-\frac{y_1^2+y_2^2}{2}&
3b+\frac{y_1z_1+y_2z_2}{2}
\\
0&1&x_2&x_1&
x_1x_2&
\frac{x_1^2+x_2^2}{2}&
-z_2-\frac{x_1y_2+x_2y_1}{2}&
-z_1-\frac{x_1y_1+x_2y_2}{2}&
3a+\frac{x_1z_1+x_2z_2}{2}
\\
0&0&1&0&
x_1&x_2&-y_1&-y_2&2z_2
\\
0&0&0&1&
x_2&x_1&-y_2&-y_1&2z_1
\\
0&0&0&0&
1&0&0&0&y_2
\\
0&0&0&0&
0&1&0&0&y_1
\\
0&0&0&0&
0&0&1&0&x_2
\\
0&0&0&0&
0&0&0&1&x_1
\\
0&0&0&0&
0&0&0&0&1
\end{pmatrix}
}
for $x_1,x_2,y_1,y_2,z_1,z_2,a,b\in\bb{R}$.
\end{rem}
\begin{prob}
Does there exist a smooth quasi-projective variety $M$ such that $\pi_1(M, x)$ is  a lattice in $G(\bb{R})$?
\end{prob}
\section{An algebraic remark}%gradingとbigradingの差を強調する表はいらなくて反例はリー代数の表示だけでいい。6次元までは
\subsection{Gap between gradings and bigradings}
In \cite{grading},  we study lattices in simply-connected nilpotent Lie groups which could be realized as the fundamental groups of smooth quasi-projective varieties by using graded Lie algebras satisfying certain conditions. The following conditions on graded Lie algebras immediately follows from \cref{W}.
\begin{cor}[\cite{Morgan, grading}]
Let $M$ be any smooth quasi-projective variety and $x\in M$ be a base point.
 We assume that  the fundamental group $\pi_{1}(M, x)$ is a lattice in a simply connected nilpotent Lie group $N$.
 Then there exists a grading $\fr{n}_{\bb{C}}=\dis{\bo_{i\leq -1}\fr{n}_i}$ such that:
\\\ \indent
{\bf (W):} For the induced grading $H^*(\fr{n}_{\bb{C}})=\dis{\bo_{i\geq 1}H_i^*(\fr{n}_{\bb{C}})}$,  we have
\ali{H^1(\fr{n}_{\bb{C}})=H_1^1(\fr{n}_{\bb{C}})\o H_2^1(\fr{n}_{\bb{C}})}
\\\ \indent
and
\ali{H^2(\fr{n}_{\bb{C}})=H_2^2(\fr{n}_{\bb{C}})\o H_3^2(\fr{n}_{\bb{C}})\o H_4^2(\fr{n}_{\bb{C}}).} 
\\\ \indent
{\bf (H):} For any odd integer $k$,  $\dim{\fr{n}_k}$ and $\dim{H_k^j}$ are even.
\end{cor}
The condition {\bf (W)} holds by taking the total grading $\fr{n}_i:=\bo_{p+q=i}\fr{n}_{p, q}$. The condition {\bf (H)} holds by the Hodge symmetry $\overline{\fr{n}_{p, q}}\equiv \fr{n}_{q, p}\mod \bo_{s+t<p+q}\fr{n}_{s, t}$.
In this paper,  we obtain sharper criteria for a lattice in a simply-connected nilpotent Lie group to be the fundamental group of a smooth quasi-projective variety. Furthermore,  our classification in dimensions up to 6 includes the results of \cite{grading}.
\begin{rem}
The following Lie algebras
\ali{
&\bullet\ L_{5, 9}=\ang{X_1, X_2, X_3, X_4, X_5}:&[X_1, X_2]=X_3, [X_2, X_3]=X_4, [X_1, X_3]=X_5. \\
&\bullet\ L_{6, 9}=L_{5, 9}\o\bb{C}. \hspace{2.1cm}&\\
&\bullet\ L_{6, 21}(-1)=\ang{X_1, X_2, X_3, X_4, X_5, X_6}:&[X_1, X_2]=X_3,  [X_2, X_3]=X_4, [X_1, X_3]=X_5, \\&[X_1, X_4]=[X_2, X_5]=X_6. \\
&\bullet\ L_{6, 24}(0)=\ang{X_1, X_2, X_3, X_4, X_5, X_6}:&[X_1, X_3]=X_4,  [X_3, X_4]=X_5, [X_1, X_4]=[X_3, X_2]=X_6.\\
&\bullet\ L_{6, 24}(1)=\ang{X_1, X_2, X_3, X_4, X_5, X_6}:&[X_1, X_2]=X_3,  [X_2, X_3]=[X_2, X_4]=X_5, [X_1, X_3]=X_6.
}
do not admit a bigrading satisfying {\bf (W)}. However,  these admit  gradings satisfying {\bf (W)} and {\bf (H)} $($ see \cite{grading}$)$. Thus the bigradings provide sharper criteria for lattices in simply-connected nilpotent Lie groups to be realized as the fundamental groups of smooth quasi-projective varieties than gradings.
\end{rem}

\end{document}